\documentclass[11pt]{article}

\usepackage[a4paper,margin=30mm]{geometry}
\usepackage[T1]{fontenc}
\usepackage[utf8]{inputenc}
\usepackage{lmodern}
\usepackage{microtype}
\usepackage{amsmath,amssymb,amsthm,mathtools}
\usepackage{enumitem}
\usepackage{booktabs}
\usepackage{hyperref}
\usepackage{comment}
\usepackage{tikz}
\usepackage[nameinlink,noabbrev]{cleveref}
\usepackage{float}
\usepackage{bm}

\hypersetup{
  colorlinks=true,
  linkcolor=blue,
  citecolor=blue,
  urlcolor=blue,
  pdftitle={Optimal Polynomial Tractability Exponents for the Inverse Star Discrepancy},
  pdfauthor={Josef Dick},
  pdfsubject={Inverse star discrepancy and polynomial tractability},
  pdfkeywords={star discrepancy, inverse discrepancy, tractability, Gram matrix}
}

\theoremstyle{plain}
\newtheorem{theorem}{Theorem}[section]

\newtheorem{lemma}[theorem]{Lemma}

\theoremstyle{remark}
\newtheorem{remark}[theorem]{Remark}
\theoremstyle{definition}

\newcommand{\1}{\mathbf{1}}
\newcommand{\N}{\mathbb{N}}
\newcommand{\R}{\mathbb{R}}
\newcommand{\vol}{\operatorname{vol}}
\newcommand{\tr}{\operatorname{tr}}
\newcommand{\rank}{\operatorname{rank}}

\title{Optimal Polynomial Tractability Exponents for the Inverse Star Discrepancy}

\author{
Josef Dick\\
School of Mathematics and Statistics\\
UNSW Sydney\\
Sydney, NSW 2052, Australia\\
\texttt{josef.dick@unsw.edu.au}
}

\date{}

\begin{document}
\maketitle

\begin{abstract}
The inverse of the star discrepancy $n^\ast(d, \varepsilon)$ satisfies
 \[
d \varepsilon^{-1} \lesssim n^{\ast}(d,\varepsilon)\lesssim d\varepsilon^{-2}
\] for all $d \in \mathbb{N}$ and $0 < \varepsilon < \varepsilon_0$. The upper bound was established by Heinrich, Novak, Wasilkowski and Wo\'{z}niakowski (2001), while the lower bound is due to Hinrichs (2004). Steinerberger (2023) subsequently gave an elementary proof of the latter result. These bounds imply that the inverse of the star discrepancy depends linearly on the dimension, but the exact exponent of $\varepsilon^{-1}$ had remained open. 
\medskip

In this paper we prove a lower bound which shows that the exponent $2$ of $\varepsilon^{-1}$ in the upper bound cannot be improved. More precisely, for every $0<\alpha<1$ and fixed $0<A\le B$,
there exist constants $c_{\alpha,B}>0$ and
$\varepsilon_{\alpha,A}>0$ such that, for every
$0<\varepsilon<\varepsilon_{\alpha,A}$ and every integer $d$
satisfying
\[
 A\varepsilon^{-\alpha}\le d\le B\varepsilon^{-\alpha},
\]
one has
\[
 n^\ast(d,\varepsilon)
 \ge c_{\alpha,B}\,d\,\varepsilon^{-(2-\alpha)}.
\]
Along these polynomial strips the right-hand side is of order $\varepsilon^{-2}$.  
\medskip

Consequently, every uniform polynomial upper estimate $n^{\ast}(d,\varepsilon)\le C d^q\varepsilon^{-p}$ must satisfy $p\ge2$.  Together with the lower bound of Hinrichs (2004), which forces $q\ge1$, this proves that the exponents $p=2$ and $q=1$ in the Heinrich--Novak--Wasilkowski--Wo\'{z}niakowski upper bound are individually optimal. In particular, this gives the optimal exponent $p^\ast=2$.
\end{abstract}

\medskip
\noindent\textbf{Keywords.} Star discrepancy; inverse discrepancy; tractability; Gram matrix; high-dimensional discrepancy.

\noindent\textbf{2020 Mathematics Subject Classification.} 11K38, 65Y20, 65C05.

\section{Introduction}\label{sec:introduction}

For \( \bm{y} = (y_1,\dots,y_d)\in[0,1]^d\), write
$
 [\bm{0}, \bm{y}):=\prod_{j=1}^d[0,y_j). 
$ For \(d\in\mathbb{N} \), write \([d]:=\{1,\dots,d\}\).
For an \(N\)-point multiset
$
 P=\{\bm{x}_1, \dots, \bm{x}_N\}\subset[0,1)^d,
$
the star discrepancy is
\begin{equation*}
 D_N^{\ast}(P)
 :=
 \sup_{ \bm{y} \in [0, 1]^d}
 \left|
 \frac1N\sum_{n=1}^N\1_{[\bm{0}, \bm{y})}(\bm{x}_n)
 -\vol([ \bm{0}, \bm{y}))
 \right|.
\end{equation*}
The minimal star discrepancy is
\begin{equation*}
 D^{\ast}(N,d)
 :=
 \inf_{P\subset[0,1)^d,\,|P|=N}D_N^{\ast}(P),
\end{equation*}
and its inverse is defined by
\begin{equation*}
 n^{\ast}(d,\varepsilon)
 :=
 \min\{N\in\N: D^{\ast}(N,d)\le\varepsilon\}.
\end{equation*}

Heinrich, Novak, Wasilkowski and Wo\'{z}niakowski \cite{HNWW01} proved that there is an absolute constant \(C_{\rm H}>0\) such that
\begin{equation}\label{eq:HNWW-upper-intro}
 n^{\ast}(d,\varepsilon)
 \le C_{\rm H}d\varepsilon^{-2},
 \qquad d\in\N,\quad 0<\varepsilon<1.
\end{equation}
Their proof proceeds through the corresponding discrepancy estimate
\[
 D^\ast(N,d)\le C\sqrt{\frac dN}
\]
for another absolute constant \(C>0\). The proof is probabilistic.  Aistleitner \cite{Ais11} subsequently obtained the first explicit value for the constant in the discrepancy formulation, showing that
\[
 D^{\ast}(N,d)\le 10\sqrt{\frac dN}.
\]
Pasing and Wei\ss{} \cite{PW20} reduced the constant to \(9\), and Gnewuch, Pasing and Wei\ss{} \cite{GPW21} obtained the explicit bound \(2.4968\). Wei\ss{} \cite{W26} recently improved the constant further to $2.4631832$.  Related probabilistic estimates for negatively dependent samples, including Latin hypercube samples, were developed by Gnewuch and Hebbinghaus \cite{GH21}.  Doerr \cite{Doe14} proved that the expected star discrepancy of \(N\) independent uniformly distributed points is of order \(\sqrt{d/N}\), so this rate is optimal up to constants for i.i.d. uniform sampling. This result was extended to Latin hypercube sampling in \cite{DDG18}.

The work of Heinrich, Novak, Wasilkowski and Wo\'{z}niakowski \cite{HNWW01} also established the lower bound
\[
 n^{\ast}(d,\varepsilon)\gtrsim d\log(\varepsilon^{-1}).
\]
Hinrichs \cite{H04} subsequently strengthened this to: there are absolute constants \(c_{\rm H}>0\) and \(\varepsilon_{\rm H}>0\) such that
\begin{equation}\label{eq:Hinrichs-lower-intro}
 n^{\ast}(d,\varepsilon)
 \ge c_{\rm H}d\varepsilon^{-1},
 \qquad d\in\N,\quad 0<\varepsilon\le\varepsilon_{\rm H}.
\end{equation}
A further combinatorial proof was given by Aistleitner and Hinrichs \cite{AH18}, and Steinerberger \cite{S23} later gave a short elementary proof.  These estimates establish the exact linear dependence on the dimension, but leave a gap between the powers \(1\) and \(2\) of \(\varepsilon^{-1}\).

The purpose of this paper is to resolve the corresponding question at the level of uniform polynomial tractability.  We show that the power \(2\) in \eqref{eq:HNWW-upper-intro} cannot be replaced by any smaller power, even if one allows an arbitrary fixed polynomial dependence on the dimension.  The key lower bound is obtained in joint regimes in which the dimension grows polynomially with \(\varepsilon^{-1}\).

We note that the optimal accuracy exponent is obtained not by strengthening the lower bound of Hinrichs uniformly over all $d$ and $\varepsilon$, but by proving sharper lower bounds in polynomial regimes of the form $d\asymp\varepsilon^{-\alpha}$. Our approach is therefore distinct from the geometric arguments of Hinrichs and Steinerberger.

The proof is based on a Gram-matrix construction.  A centred one-dimensional step function is chosen so that its tensor products over distinct coordinate sets are orthogonal with respect to Lebesgue measure.  Small star discrepancy implies that their empirical inner products over the point set are close to the corresponding Lebesgue inner products.  Evaluating these tensor products at the \(N\) points therefore produces a large family of nearly orthogonal vectors in \(\R^N\).  A trace--rank inequality then forces \(N\) to be large.

This leads to our main result.

\begin{theorem}\label{thm:main-strip}
Let \(0<\alpha<1\) and \(0<A\le B\).  Choose an integer \(m>2/\alpha\), define
\begin{equation}\label{eq:intro-Cm}
 C_m:=\frac m{64}\left(\frac94\right)^{1/m},
\end{equation}
and set
\begin{equation}\label{eq:intro-epsilon-alpha}
 \varepsilon_{\alpha,A,m}
 :=
 \min\left\{
 2^{-m-1},
 \left(\frac A{C_m}\right)^{1/(\alpha-2/m)},
 \left(\frac A m\right)^{1/\alpha}
 \right\}.
\end{equation}
Then, for every \(0<\varepsilon\le\varepsilon_{\alpha,A,m}\) and every integer \(d\) satisfying
\begin{equation}\label{eq:intro-alpha-strip}
 A\varepsilon^{-\alpha}\le d\le B\varepsilon^{-\alpha},
\end{equation}
one has
\begin{equation*}
 n^{\ast}(d,\varepsilon)
 \ge
 \frac1{8B64^m}\,
 d\,\varepsilon^{-(2-\alpha)}.
\end{equation*}
\end{theorem}

Theorem~\ref{thm:main-strip} has an immediate consequence that is stronger than a lower bound on any single prescribed curve.  Along the strip \(d\asymp\varepsilon^{-\alpha}\), its right-hand side satisfies
\[
 d\varepsilon^{-(2-\alpha)}\asymp\varepsilon^{-2}.
\]
Suppose that, for some fixed \(p,q\ge0\), a uniform estimate
\begin{equation}\label{eq:intro-hypothetical-upper}
 n^{\ast}(d,\varepsilon)\le C d^q\varepsilon^{-p}
\end{equation}
were valid for all dimensions and all sufficiently small \(\varepsilon\).  Taking \(d=\lceil\varepsilon^{-\alpha}\rceil\), Theorem~\ref{thm:main-strip} gives a lower bound of order \(\varepsilon^{-2}\), whereas \eqref{eq:intro-hypothetical-upper} gives an upper bound of order \(\varepsilon^{-(p+\alpha q)}\).  Since \(\alpha>0\) may be chosen arbitrarily small, this comparison forces \(p\ge2\).  Thus the accuracy exponent in the upper estimate from \cite{HNWW01} is optimal even when the power of \(d\) is allowed to be any fixed finite number.

To state this conclusion precisely, define the \emph{accuracy exponent}
\begin{equation*}
 p^{\ast}
 :=
 \inf\left\{
 p\ge0:
 \begin{array}{l}
 \text{there exist }C>0,\ q\ge0,\ \varepsilon_1>0\text{ such that}\\
 n^{\ast}(d,\varepsilon)\le C d^q\varepsilon^{-p}
 \text{ for all }d\in\N,\ 0<\varepsilon\le\varepsilon_1
 \end{array}
 \right\},
\end{equation*}
and the \emph{dimension exponent}
\begin{equation*}
 q^{\ast}
 :=
 \inf\left\{
 q\ge0:
 \begin{array}{l}
 \text{there exist }C>0,\ p\ge0,\ \varepsilon_1>0\text{ such that}\\
 n^{\ast}(d,\varepsilon)\le C d^q\varepsilon^{-p}
 \text{ for all }d\in\N,\ 0<\varepsilon\le\varepsilon_1
 \end{array}
 \right\}.
\end{equation*}

\begin{theorem}[Optimal polynomial tractability exponents]\label{thm:tractability-exponents}
The inverse star discrepancy satisfies
\begin{equation*}
 p^{\ast}=2
 \qquad\text{and}\qquad
 q^{\ast}=1.
\end{equation*}
Consequently, the powers of \(\varepsilon^{-1}\) and \(d\) in the general upper bound of Heinrich, Novak, Wasilkowski and Wo\'{z}niakowski \cite{HNWW01} are individually optimal.
\end{theorem}

The assertion \(q^\ast=1\) follows by combining the upper bound from
\cite{HNWW01} with the lower bound of Hinrichs \cite{H04}.
As established above, Theorem~\ref{thm:main-strip} yields \(p^\ast=2\),
thereby proving a conjecture in \cite[p.~63]{NW10}; see also \cite[Problem~3]{H03}.

\begin{remark}\label{rem:strip-optimality}
Fix \(0<\alpha<1\).  On a strip \(d\asymp\varepsilon^{-\alpha}\), Theorem~\ref{thm:main-strip} gives
\[
 n^{\ast}(d,\varepsilon)
 \gtrsim d\varepsilon^{-(2-\alpha)}
 \asymp\varepsilon^{-2}.
\]
The total power \(2\) of \(\varepsilon^{-1}\) is exactly the power needed to prove that $p^\ast \ge 2$. Since the upper bound from \cite{HNWW01} attains \(p=2\), no stronger strip exponent is required to establish \(p^{\ast}=2\).  In this precise tractability sense, Theorem~\ref{thm:main-strip} is optimal.

On the other hand, there is still a gap in the precise bound on the strip $d \asymp \varepsilon^{-\alpha}$. On these strips our lower bound shows $n^\ast(d, \varepsilon) \gtrsim \varepsilon^{-2}$, whereas the upper bound only gives
\begin{equation*}
n^\ast(d,\varepsilon) \lesssim d\varepsilon^{-2} \asymp \varepsilon^{-(2+\alpha)}.
\end{equation*}
Determining the exact mixed dependence on \(d\) and \(\varepsilon\) in individual joint regimes remains an open problem.
\end{remark}

\begin{remark}[The fixed-dimensional obstruction]\label{rem:fixed-d-obstruction}
Theorem~\ref{thm:tractability-exponents} concerns uniform polynomial exponents; it does not assert a pointwise lower bound
\[
 n^{\ast}(d,\varepsilon)\gtrsim d\varepsilon^{-2}
\]
for every pair \((d,\varepsilon)\). Indeed, such a bound is false. To see this, for every fixed \(d\), classical low-discrepancy constructions satisfy
\[
 D_N^{\ast}(P_N)
 \le C_d\frac{(\log N)^{d-1}}N,
\]
see, for example, \cite{DP10}.  Consequently,
\[
 n^{\ast}(d,\varepsilon)
 \le C'_d\varepsilon^{-1}
 \bigl(1+\log\varepsilon^{-1}\bigr)^{d-1}
\]
for sufficiently small \(\varepsilon\), and this is \(o(\varepsilon^{-2})\) for every fixed dimension. 
\end{remark}

\begin{remark}[Comparison with the lower bound of Hinrichs]\label{rem:comparison-hinrichs-gram}
The lower bound of Hinrichs,
\[
 n^{\ast}(d,\varepsilon)\gtrsim d\varepsilon^{-1},
\]
and the Gram-matrix lower bound obtained in this paper are complementary.  Neither estimate dominates the other uniformly over all pairs \((d,\varepsilon)\).

To explain the comparison, consider first a polynomial strip
\[
 A\varepsilon^{-\alpha}\le d\le B\varepsilon^{-\alpha},
 \qquad 0<\alpha<1.
\]
By Theorem~\ref{thm:main-strip} we obtain
\[
 n^{\ast}(d,\varepsilon)\gtrsim d\,\varepsilon^{-(2-\alpha)}
 \asymp \varepsilon^{-2},
\]
whereas the lower bound of Hinrichs yields only
\[
 n^{\ast}(d,\varepsilon)\gtrsim d\varepsilon^{-1}
 \asymp \varepsilon^{-(1+\alpha)}.
\]
Hence on every such strip the Gram-matrix bound is asymptotically stronger.

By contrast, if \(d\gg \varepsilon^{-1}\), then
\[
 d\varepsilon^{-1}\gg \varepsilon^{-2},
\]
so the Hinrichs bound is stronger than the Gram-matrix bound, whose strength is of order \(\varepsilon^{-2}\) in the high-dimensional polynomial-strip regime.  On the transition scale
\[
 d\asymp \varepsilon^{-1},
\]
both lower bounds are of the same order \(\varepsilon^{-2}\).

Finally, for any given $\alpha \in (0, 1)$ and for every fixed dimension $d$, condition \eqref{eq:intro-alpha-strip} does not hold for sufficiently small \(\varepsilon\), whereas Hinrichs' bound continues to give
\[
 n^{\ast}(d,\varepsilon)\gtrsim d\varepsilon^{-1}.
\]
Thus Hinrichs' bound remains stronger in fixed dimension.  It is also
stronger in the opposite high-dimensional regime \(d\gg\varepsilon^{-1}\).
For fixed \(m\), the Gram-matrix argument is effective in the intermediate
range
\[
 \varepsilon^{-2/m}\lesssim d\ll\varepsilon^{-1},
\]
up to constants depending on \(m\).  In particular, it improves the
dependence on \(\varepsilon^{-1}\) along every polynomial strip
\(d\asymp\varepsilon^{-\alpha}\) for which \(m>2/\alpha\).

Figure~\ref{fig:hinrichs-vs-gram} illustrates the regions of the $(\varepsilon^{-1}, d)$ plane where the Hinrichs bound is stronger than the Gram bound, and vice versa.
\end{remark}

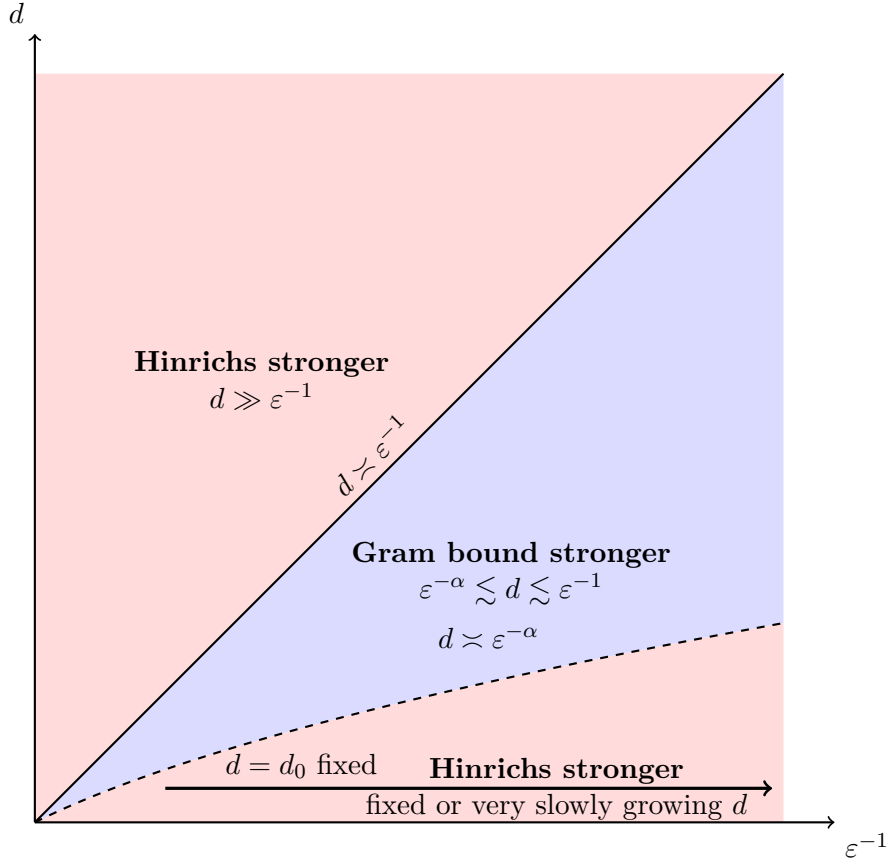
\begin{figure}[H]
\centering
\begin{tikzpicture}[x=1.5cm,y=1.5cm]

\def\xmin{1.0}
\def\ymin{1.0}
\def\xmax{7.6}
\def\ymax{7.6}


\fill[red!14]
  (\xmin,\ymin) --
  (\xmax,\ymin) --
  plot[domain=\xmax:\xmin,samples=100] (\x,{sqrt(\x)}) --
  cycle;

\fill[blue!14]
  (\xmin,\ymin) --
  plot[domain=\xmin:\xmax,samples=100] (\x,{sqrt(\x)}) --
  (\xmax,\xmax) --
  cycle;

\fill[red!14]
  (\xmin,\ymin) --
  (\xmin,\ymax) --
  (\ymax,\ymax) --
  cycle;

\draw[->,thick] (\xmin,\ymin) -- (\xmax+0.45,\ymin)
  node[below right] {$\varepsilon^{-1}$};
\draw[->,thick] (\xmin,\ymin) -- (\xmin,\ymax+0.35)
  node[above left] {$d$};


\draw[thick,dashed]
  plot[domain=\xmin:\xmax,samples=100] (\x,{sqrt(\x)});
\node[above] at (5.0,2.45)
  {$d\asymp \varepsilon^{-\alpha}$};

\draw[thick]
  (\xmin,\ymin) -- (\ymax,\ymax);
\node[rotate=45,above] at (4.1,4.1)
  {$d\asymp\varepsilon^{-1}$};


\node[align=center] at (5.6,1.3)
  {\textbf{Hinrichs stronger}\\
   fixed or very slowly growing \(d\)};

\node[align=center] at (5.2,3.2)
  {\textbf{Gram bound stronger}\\
   \(\varepsilon^{-\alpha}\lesssim d\lesssim \varepsilon^{-1}\)};

\node[align=center] at (3.0,4.9)
  {\textbf{Hinrichs stronger}\\
   \(d\gg \varepsilon^{-1}\)};

\draw[very thick,->] (2.15,1.30) -- (7.5,1.30);
\node[above] at (3.35,1.30) {$d=d_0$ fixed};

\end{tikzpicture}

\caption{Schematic comparison, for a fixed $0 < \alpha < 1$, of the
Hinrichs lower bound and the Gram-matrix lower bound.  The Gram-matrix
estimate can improve on Hinrichs' lower bound in the intermediate region
\(\varepsilon^{-\alpha}\lesssim d\lesssim\varepsilon^{-1}\).
Below the lower transition curve, including every fixed-dimensional
horizontal ray \(d=d_0\) for sufficiently small \(\varepsilon\), the
Hinrichs lower bound is eventually stronger. Hinrichs' lower bound is also
stronger above the transition scale \(d\asymp\varepsilon^{-1}\).
Constants and the factor \(64^m\) are suppressed in this schematic
comparison.}
\label{fig:hinrichs-vs-gram}
\end{figure}

The remainder of the paper is organised as follows.  Section~\ref{sec:proof-main} proves Theorem~\ref{thm:main-strip}.  The short proof of Theorem~\ref{thm:tractability-exponents} is given in Section~\ref{sec:proof-tractability}.

\section{Proof of the polynomial-strip lower bound}\label{sec:proof-main}

We prove Theorem~\ref{thm:main-strip} through a single Gram-matrix estimate.
For \(P=\{ \bm{x}_1, \dots, \bm{x}_N\}\subset[0,1)^d\), write
\[
 D:=D_N^\ast(P).
\]

\begin{lemma}[Anchored-box transfer]\label{lem:transfer-short}
Suppose that
\[
 f=c_0+\sum_{\ell=1}^L c_\ell\1_{A_\ell},
 \qquad A_\ell=[\bm{0}, \bm{y}^{(\ell)})\subset[0,1)^d.
\]
Then
\[
 \left|
 \frac1N\sum_{n=1}^N f(\bm{x}_n)
 -\int_{[0,1)^d}f(\bm{x})\,d\bm{x}
 \right|
 \le D\sum_{\ell=1}^L|c_\ell|.
\]
\end{lemma}

\begin{proof}
The constant term cancels, and for every \(\ell\),
\[
 \left|
 \frac1N\sum_{n=1}^N\1_{A_\ell}(\bm{x}_n)-\vol(A_\ell)
 \right|\le D.
\]
The result follows from linearity and the triangle inequality.
\end{proof}

For \(j\in[d]\), put $I_j(\bm{x}):=\1_{[0,1/3)}(x_j)$, and define the base-$3$ Haar function
\[ h(t):=\frac{3\1_{[0,1/3)}(t)-1}{\sqrt2}.
\]
The elementary identities
\begin{equation}\label{eq:h-short}
 \int_0^1h(t)\,dt=0,\qquad
 \int_0^1h(t)^2\,dt=1,\qquad
 h(t)^2=\frac32\1_{[0,1/3)}(t)+\frac12
\end{equation}
will be used repeatedly.  The sums of the absolute values of the
coefficients in the expansions of \(h\) and \(h^2\) are, respectively,
\[
 | 3/\sqrt{2}| + |-1/\sqrt{2}| = 2\sqrt2
 \qquad\text{and}\qquad
 |3/2| + |1/2| = 2.
\]
Moreover, for every \(U\subseteq[d]\),
\[
 \prod_{j\in U}I_j(\bm{x})=\1_{[\bm{0}, \bm{y}(U))}(\bm{x}),
 \qquad
 \bm{y}(U)_j=
 \begin{cases}
 1/3,&j\in U,\\
 1,&j\notin U.
 \end{cases}
\]
Thus every indicator monomial appearing below is the indicator of an
anchored box.

\begin{lemma}[Gram-matrix estimate]\label{lem:gram-short}
Let \(1\le m\le d\), set \(M=\binom dm\), and assume that
\[
 D\le2^{-m-1}.
\]
Then
\begin{equation*}
 N\ge\frac{M}{9+4M64^mD^2}.
\end{equation*}
\end{lemma}

\begin{proof}
Let
\[
 \mathcal F_m:=\{S\subseteq[d]:|S|=m\},
 \qquad
 \phi_S(\bm{x}):=\prod_{j\in S}h(x_j),
\]
and define
\[
 v_S:=\frac1{\sqrt N}
 \bigl(\phi_S(\bm{x}_1),\dots,\phi_S(\bm{x}_N)\bigr)\in\R^N.
\]
Let \(G=(G_{S,T})_{S,T\in\mathcal F_m}\) be their Gram matrix:
\[
 G_{S,T}
 =
 \langle v_S,v_T\rangle
 =
 \frac1N\sum_{n=1}^N\phi_S(\bm{x}_n)\phi_T(\bm{x}_n).
\] 
Note that $G$ is real, symmetric, and positive semidefinite.
Since the vectors \(v_S\) lie in \(\R^N\),
\begin{equation}\label{eq:rank-upper-short}
 \rank(G)\le N.
\end{equation}

For \(S\in\mathcal F_m\), the product expansion of
\[
 \phi_S(\bm{x})^2
 =
 \prod_{j\in S}
 \left(\frac32I_j(\bm{x})+\frac12\right) = \sum_{U \subseteq S} \prod_{j \in U} \frac32 I_j(\bm{x}) \prod_{j \in S \setminus U} \frac12
\]
has coefficient absolute-value sum \(\prod_{j \in S} (|3/2| + |1/2|) = 2^m\).  By
\eqref{eq:h-short}, its Lebesgue integral equals \(1\).  Hence
Lemma~\ref{lem:transfer-short} gives
\begin{equation}\label{eq:diag-short}
 |G_{S,S}-1| = \left|\frac{1}{N} \sum_{n=1}^N \phi_S(\bm{x}_n)^2 - \int_{[0,1)^d} \phi_S(\bm{x})^2 \,d \bm{x} \right| \le 2^m D.
\end{equation}

Now let \(S\ne T\), write \(s=|S\cap T|\) and \(r=m-s\). Let $S \triangle T$ denote the symmetric difference. Then
\begin{align*}
 \phi_S\phi_T
 & =
 \prod_{j\in S\triangle T}h(x_j)
 \prod_{j\in S\cap T}h(x_j)^2  \\ & = \sum_{V \subseteq S \triangle T} \sum_{U \subseteq S \cap T} \left(\prod_{j \in V} \frac{3}{\sqrt{2}} I_j(\bm{x}) \right) \left( -\frac{1}{\sqrt{2}} \right)^{2r - |V|} \left( \prod_{i \in U} \frac32 I_i(\bm{x}) \right) \left( \frac12 \right)^{s - |U|}.
\end{align*}
Since \(S\triangle T\ne \emptyset \), one factor has mean zero, and
therefore
\[
 \int_{[0,1)^d}\phi_S(\bm{x})\phi_T(\bm{x})\,d \bm{x} = 0.
\]
The coefficient absolute-value sum of the product expansion is
\[
 (2\sqrt2)^{2r}2^s
 =2^m4^r
 \le8^m.
\]
Another application of Lemma~\ref{lem:transfer-short} yields
\begin{equation}\label{eq:offdiag-short}
 |G_{S,T}| = \left| \frac{1}{N} \sum_{n=1}^N \phi_S(\bm{x}_n) \phi_T(\bm{x}_n) - \int_{[0,1)^d} \phi_S(\bm{x}) \phi_T(\bm{x}) \,d \bm{x} \right| \le 8^m D
 \qquad(S\ne T).
\end{equation}

We use the elementary trace--rank inequality
\begin{equation}\label{eq:trace-rank-short}
 \rank(G)\ge\frac{(\tr G)^2}{\tr(G^2)}.
\end{equation}
Indeed, if \(\lambda_1,\dots,\lambda_q>0\) are the nonzero eigenvalues
of \(G\), then Cauchy--Schwarz gives
\[
 (\tr G)^2
 =
 \left(\sum_{j=1}^q\lambda_j\right)^2
 \le q\sum_{j=1}^q\lambda_j^2
 =
 \rank(G)\tr(G^2).
\]

Since \(2^mD\le1/2\), \eqref{eq:diag-short} implies
\[
 \tr G = \sum_{S \in \mathcal{F}_m } G_{S,S} \ge M (1 - 2^m D) \ge \frac M2.
\]
Furthermore, by \eqref{eq:diag-short} we have $|G_{S,S}| \le 1 + 2^m D \le 3/2$ and using \eqref{eq:offdiag-short},
\[
 \tr(G^2)
 =
 \sum_{S,T\in\mathcal F_m}G_{S,T}^2
 \le
 M\left(\frac32\right)^2
 +M(M-1)64^mD^2
 \le
 \frac94M+M^2 64^mD^2.
\]
Combining this estimate with \eqref{eq:rank-upper-short} and
\eqref{eq:trace-rank-short}, we obtain
\[
 N \ge \rank(G)
 \ge
 \frac{M^2/4}{(9/4)M+M^2 64^mD^2}
 =
 \frac{M}{9+4M64^mD^2}.
\]
\end{proof}

\begin{proof}[Proof of Theorem~\ref{thm:main-strip}]
Fix \(m>2/\alpha\) and let \(C_m\) and
\(\varepsilon_{\alpha,A,m}\) be defined by
\eqref{eq:intro-Cm} and \eqref{eq:intro-epsilon-alpha}.  Let
\(0<\varepsilon\le\varepsilon_{\alpha,A,m}\) and suppose that
\[
 A\varepsilon^{-\alpha}\le d\le B\varepsilon^{-\alpha}.
\]
The definition of \(\varepsilon_{\alpha,A,m}\) gives $\varepsilon \le (A/m)^{1/\alpha}$ and therefore
\[
 d\ge A\varepsilon^{-\alpha}\ge m
\]
and, since \(\alpha-2/m>0\), we obtain from $\varepsilon \le (A/C_m)^{1/(\alpha - 2/m)}$ that
\[
 d\ge A\varepsilon^{-\alpha}
 \ge C_m\varepsilon^{-2/m}.
\]
Using
\[
 \binom dm
 =
 \prod_{j=0}^{m-1}\frac{d-j}{m-j}
 \ge\left(\frac dm\right)^m,
\]
we obtain
\begin{equation}\label{eq:M-lower-short}
 M:=\binom dm
 \ge
 \left(\frac{C_m}{m}\right)^m\varepsilon^{-2}
 =
 \frac9{4\cdot64^m\varepsilon^2}.
\end{equation}

Let \(P\) be any \(N\)-point set with \(D_N^\ast(P)\le\varepsilon\).
Because \(\varepsilon\le2^{-m-1}\), Lemma~\ref{lem:gram-short} applies
with \(D=D_N^\ast(P)\), and therefore
\[
 N
 \ge
 \frac{M}{9+4M64^m\varepsilon^2}.
\]
Put \(t=M64^m\varepsilon^2\).  By \eqref{eq:M-lower-short},
\(t\ge9/4\), and hence
\[
 \frac{M}{9+4M64^m\varepsilon^2}
 =
 \frac1{64^m\varepsilon^2}\frac{t}{9+4t}
 \ge
 \frac1{8\cdot64^m\varepsilon^2}.
\]
Taking the minimum over all such \(N\) gives
\[
 n^\ast(d,\varepsilon)
 \ge
 \frac1{8\cdot64^m}\varepsilon^{-2}.
\]
Finally, \(d\le B\varepsilon^{-\alpha}\) implies
\[
 \varepsilon^{-2}
 \ge
 \frac1B d\varepsilon^{-(2-\alpha)}.
\]
Combining the last two inequalities proves
\[
 n^\ast(d,\varepsilon)
 \ge
 \frac1{8B64^m}
 d\varepsilon^{-(2-\alpha)}.
\]
\end{proof}

\begin{remark}[Orthogonality and the continuous Gram matrix]
\label{rem:continuous-gram}
The use of orthogonality in discrepancy lower bounds goes back to
Roth's classical proof of the \(L_2\)-discrepancy estimate
\cite{Roth54}; see also the surveys \cite{Bil11,Bil14}. Roth's method combines many local discrepancy contributions by means
of orthogonal, usually dyadic Haar-type, test functions.  The present
argument is related in spirit, but the indexing of the orthogonal
systems is different: Roth's functions are indexed by dyadic boxes and scales, whereas the family used here is indexed by coordinate subsets. 

The Gram-matrix argument may be viewed as a comparison between a
continuous orthonormal system and its discretisation by the point set
\(P\). With the notation from the proof of Lemma~\ref{lem:gram-short}, for \(S,T\in\mathcal F_m\),
\[
 \int_{[0,1)^d}\phi_S(\bm{x})\phi_T(\bm{x})\,d\bm{x} = \delta_{S,T},
\]
so the continuous Gram matrix is the identity \(I_M\), of rank
\(M=\binom dm\).  The empirical Gram matrix \(G\) replaces this
integral by the average over the \(N\) points of \(P\).  Since the
corresponding vectors lie in \(\mathbb R^N\), one always has
\(\rank(G)\le N\).

The estimates in the proof give, before simplifying the constants,
\[
 \tr G\ge M(1-2^mD)
\]
and
\[
 \tr(G^2)
 \le
 M(1+2^mD)^2+M(M-1)64^mD^2.
\]
Consequently,
\[
 \rank(G)
 \ge
 \frac{
 M(1-2^mD)^2
 }{
 (1+2^mD)^2+(M-1)64^mD^2
 }.
\]
For fixed \(d\) and \(m\), the right-hand side tends to \(M\) as
\(D\to0\).  Since \(\rank(G)\) is an integer not exceeding \(M\), the displayed
lower bound implies \(\rank(G)=M\) once its right-hand side exceeds
\(M-1\).  Thus, for fixed \(d\) and \(m\), sufficiently small
discrepancy forces \(N\ge M\).

When \(M\) grows simultaneously with \(D^{-1}\), however, an entrywise
error bound tending to zero need not be small enough, relative to \(M\),
to guarantee full rank.  The accumulated contribution of the
\(M(M-1)\) off-diagonal entries must also be controlled.  The
trace--rank inequality provides precisely this quantitative control.
Closely related rank arguments for perturbed identity matrices appear
in Alon's work; see \cite[Lemma~2.2, p.~4]{Alon09} and
\cite[Theorem~8.1 and the following discussion, p.~13]{Alon09}.
Although rank itself is not a continuous quantity, the ratio
\[
 \frac{(\tr G)^2}{\tr(G^2)}
\]
gives a lower bound on the number of directions that remain effectively
orthogonal and hence a lower bound on \(N\).

A related use of coordinate threshold indicators appears in the work
of Tr\"odler, Volec and Vyb\'{i}ral~\cite{TVV25}, who obtain lower
bounds for the inverse dispersion via cover-free families. They
associate to a point set the coordinate threshold sets
\(F_j(u)=\{n\in[N]:x_{n,j}<u\}\), and use the requirement that suitable
axis-parallel boxes contain a point to establish cover-free
conditions on this family. Bounds from extremal set theory then
give lower bounds on the number of points. 

The Gram-matrix argument can also be viewed as a quantitative
analogue of the classical proof of Rao's bound for orthogonal
arrays; see~\cite[Theorem~2.1]{HSS99}. For a binary orthogonal array with \(N\) rows, \(d\) columns and strength \(2m\), the products of
at most \(m\) distinct coordinate signs, including the constant
function, give \(\sum_{j=0}^{m}\binom dj\) mutually orthogonal
nonzero vectors in \(\mathbb R^N\). Consequently,
\(N\ge\sum_{j=0}^{m}\binom dj\). In the present setting, the centred coordinate products are orthonormal with respect to Lebesgue measure, and small star discrepancy ensures approximate orthogonality after evaluation at the sample points.
\end{remark}

\section{Proof of the tractability statement}\label{sec:proof-tractability}

\begin{proof}[Proof of Theorem~\ref{thm:tractability-exponents}]
The upper bound \eqref{eq:HNWW-upper-intro} gives
\(p^\ast\le2\) and \(q^\ast\le1\).

To prove \(q^\ast\ge1\), suppose that
\[
 n^\ast(d,\varepsilon)\le C d^q\varepsilon^{-p}
\]
holds uniformly with \(q<1\).  Fix
\(0<\bar\varepsilon\le\varepsilon_{\rm H}\).  Combining this estimate
with \eqref{eq:Hinrichs-lower-intro} gives
\[
 c_{\rm H}d\bar\varepsilon^{-1}
 \le
 C d^q\bar\varepsilon^{-p}
 \qquad(d\in\N),
\]
which is impossible as \(d\to\infty\).

To prove \(p^\ast\ge2\), suppose that the same type of uniform estimate
holds with \(p<2\) and some finite \(q\ge0\).  Choose
\(\alpha\in(0,1)\) so small that
\[
 p+\alpha q<2,
\]
and put \(d_\varepsilon=\lceil\varepsilon^{-\alpha}\rceil\).  For
sufficiently small \(\varepsilon\),
\[
 \varepsilon^{-\alpha}
 \le d_\varepsilon\le2\varepsilon^{-\alpha}.
\]
Applying Theorem~\ref{thm:main-strip} with \(A=1\) and \(B=2\) gives
\[
 n^\ast(d_\varepsilon,\varepsilon)\ge c_\alpha\varepsilon^{-2},
\]
whereas the assumed upper estimate gives
\[
 n^\ast(d_\varepsilon,\varepsilon)
 \le C2^q\varepsilon^{-(p+\alpha q)}.
\]
This is impossible as \(\varepsilon\downarrow0\), because
\(p+\alpha q<2\).  Therefore \(p^\ast\ge2\), and the theorem follows.
\end{proof}

\section*{Acknowledgements}

The author would like to thank Michael Gnewuch, Jan Vyb\'{i}ral and Christian Wei\ss{} for helpful comments.

The work of the author was supported in part by a grant of access to OpenAI models through the ChatGPT for Academic Researchers program.

\paragraph{Declaration of generative AI use}

The author used ChatGPT 5.6 Sol for literature searches, exploratory
development, and the preparation of portions of the exposition and
LaTeX source. All mathematical arguments, calculations, references, and conclusions were independently checked and verified by the author, who takes full responsibility for the contents of the paper.

\end{document}